\documentclass[english,a4paper, abstract = true]{scrartcl}

\usepackage{mypackages}
\usepackage{setspace}
\usetikzlibrary{matrix}
\usetikzlibrary{calc}
\usetikzlibrary{hobby}

\begin{document}
\title{Harmonic morphisms and minimal submanifolds}
\author{Oskar Riedler}

\maketitle
\begin{abstract}
Harmonic morphisms, maps which preserve Laplace's equation, are intimately connected to the topic of minimal submanifolds. In this article we first characterise harmonic morphisms between Riemannian manifolds as the weakly horizontally conformal maps that preserve the equation for minimal submanifolds of co-dimension $2$. We further derive additional reduction properties of harmonic morphisms for minimal submanifolds of other co-dimensions. These theorems are then applied in an example case, yielding a novel family of degree $4$ area-minimising hypercones in $\mtr^m$, $m\geq32$.
\end{abstract}

\section{Introduction}

A map $\varphi\colon(M,g)\to(N,h)$ between two Riemannian manifolds is called a \emph{harmonic morphism} if it preserves Laplace's equation, i.e. for any locally defined function germ $f:N\to\mtr$ one has
$$\Delta_N f=0\implies \Delta_M(f\circ\varphi)=0.$$
Here $\Delta_M,\Delta_N$ denote the Laplacians of $M, N$ respectively.
\medbreak

While the above definition is purely analytic and originates in potential theory \cite{const-corn-65}, the geometric theory of harmonic morphisms turns out to be very rich. The investigation of harmonic morphisms in Riemannian geometry begins with the work of Fuglede \cite{fuglede-78} and Ishihara \cite{ishihara-79}, who found a geometric characterisation of these objects. They showed that a map is a harmonic morphism if and only if it is both \emph{weakly horizontally conformal} and \emph{harmonic}.

The notion of a harmonic morphism has turned out to be quite fertile and has since been studied in many contexts. We mention the settings of Riemannian polyhedra \cite{eells-fuglede-01}, metric graphs \cite{urakawa-00,abbr-15}, Weyl spaces \cite{loubeau-pantilie-06,loubeau-pantilie-10}, and probability theory \cite{bcd-79,cfb-90}.
\medbreak

In this article we will be concerned with the connections between harmonic morphisms and \emph{minimal submanifolds}. In particular we wish to use harmonic morphisms to generate such submanifolds. In the case that the co-domain $N$ is a surface a celebrated theorem of Baird and Eells \cite{baird-eells-80} shows that $\varphi$ fibres the regular points $M\setminus\{ d\varphi = 0\}$ by minimal co-dimension $2$ submanifolds:

\begin{theorem}[\cite{baird-eells-80}]
Let $\varphi\colon(M,g)\to(N,h)$ be a submersive horizontally conformal map with $\dim(N)=2$. The following are equivalent:
\begin{enumerate}[label=(\roman*)]
\item $\varphi$ is harmonic and thus a harmonic morphism.
\item $\varphi^{-1}(\{p\})$ is a minimal co-dimension $2$ submanifold of $M$ for all $p\in N$.
\end{enumerate}
\end{theorem}

It is natural to ask if a similar characterisation using minimal submanifolds is possible also for higher dimensional co-domains. In the following we characterise harmonic morphisms as those weakly horizontally maps preserving the equation for minimality of co-dimension $2$ submanifolds at its regular locus.

\begin{theorem}\label{thm: co-dim-2}
Let $\varphi\colon(M,g)\to(N,h)$ be a submersive horizontally conformal map. The following are equivalent:
\begin{enumerate}[label=(\roman*)]
\item $\varphi$ is harmonic and thus a harmonic morphism.
\item For all minimal co-dimension $2$ submanifolds $B\subset N$ the pre-image $\varphi^{-1}(B)\subset M$ is a minimal co-dimension $2$ submanifold.
\end{enumerate}
\end{theorem}

Harmonic morphisms are then closely tied to the geometry of co-dimension $2$ submanifolds in the co-domain. We are also interested in submanifolds of other dimensions, in particular in hypersurfaces, i.e. co-dimension $1$ submanifolds. In this direction Baird and Gudmundsson \cite{baird-gudmundsson-92} derive results for so-called $p$-harmonic morphisms. Specialising to the case of a harmonic morphism, their statement treats again maps to a surface.

\begin{theorem}[cf. Corollary 3.4 of \cite{baird-gudmundsson-92}]
Let $\varphi\colon(M,g)\to(N,h)$ be a submersive harmonic morphism with $\dim(N)=2$ and let $B\subset N$ be a minimal submanifold of co-dimension $1$ (i.e.\ a geodesic). The following are equivalent:
\begin{enumerate}[label=(\roman*)]
\item $d\varphi(\nabla\|d\varphi\|^2(x)\,)$ is tangent to $B$ for all $x\in\varphi^{-1}(B)$.
\item $\varphi^{-1}(B)$ is a minimal co-dimension $1$ submanifold of $M$.
\end{enumerate}
\end{theorem}

We extend this to arbitrary co-domains as follows:

\begin{theorem}\label{thm: co-dim-p}
Let $\varphi\colon(M,g)\to(N,h)$ be a submersive harmonic morphism and let $B\subset N$ be a minimal submanifold of co-dimension $p$ with $p\neq2$. The following are equivalent:
\begin{enumerate}[label=(\roman*)]
\item $d\varphi(\nabla\|d\varphi\|^2(x)\,)$ is tangent to $B$ for all $x\in\varphi^{-1}(B)$.
\item $\varphi^{-1}(B)$ is a minimal co-dimension $p$ submanifold of $M$.\end{enumerate}
\end{theorem}
\medbreak

We remark that these theorems can be viewed as reduction theorems for minimal submanifolds, in the spirit of \cite{hsiang-lawson-71,palais-79} using group actions and \cite{palais-terng-86} using Riemannian submersions. From this perspective they are useful tools for generating new examples of minimal submanifolds. We mention that the main technical step in showing our statements, \cref{lemma: taup-B} below, can be applied to a wider class of settings and reduction problems than indicated by the above theorems (such as prescribed mean curvature problems in the setting of $p$-harmonic morphisms).

In the second part of the paper we use the above results in order to find a novel family of area-minimising cones in $\mtr^m$. To do this we will be finding minimal cones by applying \cref{thm: co-dim-p} to \emph{polynomial harmonic morphisms} between two euclidean spaces. Polynomial harmonic morphisms are an important class of harmonic morphisms, interesting in their own right \cite{baird-83,eells-yiu-95,ABB-99,riedler-polynomials-23}. Since much about their general structure or classification still remains unknown, we mainly restrict ourselves to degree two polynomials, classified by Ou and Wood \cite{ou-wood-96,ou-97}.

This classification makes use of so-called Clifford systems, see \cref{def: clifford-system} below for their definition. Clifford systems seem well-adapted to the construction of area-minimisers, see \cite{cui-26} for recent work constructing area-minimising co-dimension $2$ cones. They also appear in the $g=4$ case for the area-minimisers constructed by Wang \cite{wang-minimal}. We show:

\begin{theorem}\label{thm: Scone-intro}
Let $m,n\in\mtn$, $m\geq32$ and $n\geq4$. For $A_1,...,A_{2n}\in M_{m\times m}(\mtr)$ a Clifford system, the cone
$$C^4_{m,n}\defeq\left\{x\in\mtr^m\ \middle|\ \sum_{i=1}^n \langle x,A_ix\rangle^2=\sum_{i=1}^n \langle x,A_{n+i}x\rangle^2\right\}$$
is an area-minimising hypersurface in $\mtr^m$.
\end{theorem}
\begin{remark}
For all values of $(m,n)$ the mean curvature of $C^4_{m,n}$ vanishes at its regular points, i.e.\ $C^4_{m,n}\setminus \{x\in\mtr^m \mid \langle x, A_ix\rangle = 0\ \forall i\in\{1,...,2n\}\,\}$ is a minimal conical submanifold of $\mtr^m$.
\end{remark}

\begin{example}
We give the examples $C^4_{8,2}$, $C^4_{16,4}$, and $C^4_{32,4}$ explicitly. Up to an isometry of the domain these cones are uniquely determined by the values of $m,n$.
\begin{gather*}
C^4_{8,2}=\bigl\{(z_1,z_2,u_1,u_2)\in\mtc^4\,\big|\, \|z_1 u_1+z_2u_2\|^2 = \|\overline{z_2}z_1-\overline{u_2}u_1 \|^2 \bigr\},\\
C^4_{16,4}=\bigl\{(z_1,z_2,u_1,u_2)\in\mth^4\,\big|\,\|z_1 u_1+z_2u_2\|^2 = \| \overline{z_2}z_1-\overline{u_2} u_1\|^2 \bigr\},\\
C^4_{32,4}=\biggl\{(z_1,z_2,z_3,z_4,u_1,u_2,u_3,u_4)\in\mth^8\ \bigg|\ \Bigl\|\sum_{i=1}^4 z_iu_i\Bigr\|^2 = \Bigl\|\sum_{i=1}^2\overline{z_{2i}}z_{2i-1}-\overline{u_{2i}}u_{2i-1}\Bigr\|^2\biggr\}.
\end{gather*}
Here the multiplications and conjugations are in the appropriate (skew-) field $\mtc$ or $\mth$.
\end{example}
\begin{remark}
Of the listed examples $C^4_{8,2}$ is not an area-minimiser, $C^4_{32,4}$ however is by \cref{thm: Scone-intro}. For $C^4_{16,4}$ it remains open as to whether or not it minimises area -- indeed $C^4_{16,4}$ is the only cone in the family $C^4_{m,n}$ for which it is not clear if it is an area-minimiser or not (cf.\ \cref{rem: area-min} in the main text).
\end{remark}

\subsection*{Structure of the article}
In \cref{sec: prelim} we review preliminary definitions and facts that we will use later. In \cref{sec: reduc} we prove the reduction theorems. In \cref{sec: area-min} we apply these reduction theorems to polynomial harmonic morphisms and prove \cref{thm: Scone-intro}.

In \cref{app: exist} we prove that for $(M,g)$ a Riemannian manifold, $x\in M$, and $V\subset T_xM$ a subvectorspace, there exists a minimal submanifold $B$ of $M$ with $x\in B$ and $T_xB=V$. This is a standard fact from Riemannian geometry, which we require for proving \cref{thm: co-dim-2}, but for which we have not been able to locate an explicit reference.

\subsection*{Acknowledgements}
Funded by the European Union (ERC Starting Grant 101116001 – COMSCAL). Views and opinions expressed are however those of the author(s) only and do not necessarily reflect those of the European Union or the European Research Council. Neither the European Union nor the granting authority can be held responsible for them.\\
This paper contains results of the author's PhD thesis at the University M\"unster. The author gratefully acknowledges the support of the Deutsche Forschungsgemeinschaft (DFG, German Research Foundation) -- Germany’s Excellence Strategy EXC 2044 390685587, Mathematics M\"unster: Dynamics-Geometry-Structure.

\section{Preliminaries}\label{sec: prelim}

In this section we review basic definitions and relations that we will use. For an overview on the general theory and literature related to harmonic morphisms we recommend both the book \cite{baird-wood-book} and the regularly updated online bibliography \cite{hm-bib}.

\begin{defn}
A map $\varphi\colon(M,g)\to(N,h)$ between two Riemannian manifolds is called a \emph{harmonic morphism} if for every $U\subseteq N$ open and $f\colon U\to\mtr$ harmonic the composition $f\circ\varphi\colon\varphi^{-1}(U)\to\mtr$ is again harmonic.
\end{defn}
Originally this class of maps was investigated in the potential-theoretic setting of Brelot harmonic spaces \cite{const-corn-65}. The geometric theory begins with the following theorem, found independently by Fuglede and Ishihara:
\begin{theorem}[\cite{fuglede-78,ishihara-79}]\label{thm: hm-char}
A map $\varphi\colon(M,g)\to(N,h)$ is a harmonic morphism if and only if it is both weakly horizontally conformal and harmonic.
\end{theorem}
We now define the terms appearing in \cref{thm: hm-char}.
\begin{defn}
Let $\varphi\colon(M,g)\to(N,h)$ be a smooth map between two Riemannian manifolds. 
\begin{enumerate}
\item $\varphi$ is \emph{weakly horizontally conformal} if for all points $x\in M$ one has either $d\varphi(x)=0$ or that $d\varphi(x)$ maps the horizontal space $\ker(d\varphi)^\perp$ conformally onto $T_{\varphi(x)}N$.
\item $\varphi$ is \emph{harmonic} if it is a stationary point of the energy functional $\varphi\mapsto E_2(\varphi)=\frac12\int_M \|d\varphi\|^2\,d\Vol_g$. Here $\|d\varphi\|^2=\Tr(d\varphi^Td\varphi)$ denotes the Hilbert-Schmidt norm of $d\varphi$.
\end{enumerate}
\end{defn}
\begin{remark}
\begin{enumerate}
\item If $\varphi\colon(M,g)\to(N,h)$ is weakly horizontally conformal there eixsts a function $\lambda\colon M\to\mtr$, called the \emph{conformality factor}, so that $\lambda^2 g(v,w)=h(d\varphi\, v,d\varphi\, w)$ for all horizontal vectors $v,w\in\ker(d\varphi)^\perp$.
\item If $e_i$ denotes a local orthonormal frame on $M$, then $\varphi\colon(M,g)\to(N,h)$ is harmonic if and only if the \emph{tension field}
\begin{equation}
\tau(\varphi)\defeq\Tr(\nabla d\varphi)=\sum_{i}\nabla^{\varphi^*}_{e_i}(d\varphi(e_i))-\sum_id\varphi(\nabla_{e_i}^Me_i)\label{eq: def-tau}
\end{equation}
vanishes, here $\nabla^{\varphi^*}$ denotes the Levi-Civita connection of $N$ pulled back to $\varphi^*(TN)$ and $\nabla^M$ the Levi-Civita connection of $M$.
\end{enumerate}
\end{remark}

Another relevant operator we will make use of is the $p$-tension field, it generalises the tension field and is the variation of the functional $E_p(\varphi)=\frac12\int_M\|d\varphi\|^p\,d\Vol_g$.
\begin{defn}
Let $\varphi\colon(M,g)\to(N,h)$ be a smooth map between Riemannian manifolds and $p\in\mtr_{>0}$. The \emph{$p$-tension field} of $\varphi$ is the section of $\varphi^*(TN)$ given by
\begin{equation}
\tau_p(\varphi)\defeq\|d\varphi\|^{p-2}(\tau(\varphi)+(p-2)d\varphi(\nabla\ln\|d\varphi\|)\,).\label{eq: p-tension-def}
\end{equation}
\end{defn}

We give some basic examples of harmonic morphisms.
\begin{example}
\begin{enumerate}
\item Riemannian submersions with minimal fibres are harmonic morphisms. Classical examples are the Hopf fibrations $S^{2m-1}\to S^m$, $m\in\{1,2,4,8\}$, quotient maps $G/H\to G/K$ for $H\subset K$ closed subgroups of a Lie group $G$ equipped with a bi-invariant metric, and the projection $(TM,g\subs{Sasaki})\to(M,g)$ where $(M,g)$ is a Riemannian manifold and $g\subs{Sasaki}$ is the Sasaki metric on the tangent bundle.
\item For $(M,g)$, $(N,h)$ Riemannian manifolds and $f\colon M\to\mtr$ smooth, the projection of the warped product $(M\times N, g+f^2h)\to (N,h)$ is a harmonic morphism. Classical examples include quotient maps to projective spaces $\mtk^{m+1}\setminus\{0\}\to\mathds{KP}^m$ for $\mtk\in\{\mtr,\mtc,\mth\}$ and the projection $S^{m+1}\setminus\{(\pm1,0,...,0)\}\to S^m$ to the latitudes.
\end{enumerate}
\end{example}

Many specialised examples and constructions exist. In \cref{sec: area-min} we will be using polynomial harmonic morphisms in order to apply our reduction theorems. These form a rich and important class of harmonic morphisms, see also Chapter 5 of \cite{baird-wood-book}. 
The cases of co-domain $\mtr^2$ ($\cong\mtc$) and $\mtr^n$, $n\geq3$, are very different.

Polynomial harmonic morphisms have a special role in the general theory of harmonic morphisms, since a general harmonic morphism is modelled by a polynomial harmonic morphism at its critical points:

\begin{theorem}[\cite{fuglede-78}]
Let $\varphi\colon(M,g)\to(N,h)$ be a harmonic morphism, not constant on any connected component of $M$. Let $x\in M$ be a critical point of $\varphi$, then the order of $\varphi$ at $x$ is finite and the symbol of $\varphi$ at $x$ is a homogeneous polynomial harmonic morphism $T_xM\to T_{\varphi(x)}N$.
\end{theorem}

In this article we will work with polynomial harmonic morphisms of degree $2$, which are classified by work of Ou and Wood \cite{ou-wood-96,ou-97}. Few examples with co-domain $\mtr^n$, $n\geq3$, and degree $d\geq3$ are known, although some general constructions using Clifford systems are possible.

We close this section by noting the following statements about harmonic morphisms to euclidean spaces:

\begin{theorem}[\cite{ABB-99}]\label{thm: poly-hm}
\begin{enumerate}
\item If $n\geq3$ then every harmonic morphism from $\mtr^m$ to $\mtr^n$ is a polynomial mapping.
\item If $n\geq2$ then every weakly horizontally conformal polynomial map $P:\mtr^m\to\mtr^n$ is harmonic and thus a harmonic morphism.
\end{enumerate}
\end{theorem}
\begin{lemma}[\cite{fuglede-78}]\label{lemma: whc-euc}
Let $\varphi\colon(M,g)\to\mtr^n$ be smooth. Then $\varphi=(\varphi_1,...,\varphi_n)$ is weakly horizontally conformal if and only if
$$\langle \nabla\varphi_i(x),\nabla\varphi_j(x)\rangle=\delta_{ij}\,\|\nabla\varphi_i(x)\|^2$$
for all $i,j\in\{1,...,n\}$.
\end{lemma}

\section{Reduction theorems}\label{sec: reduc}
In order to show the reduction theorems given in the introduction we first define what it means for a map to be weakly horizontally conformal to first order at some point:
\begin{defn}\label{def: inf-whc}
Let $(M,g)$, $(N,h)$ be two Riemannian manifolds, $x\in M$. A submersive map $\varphi:M\to N$ is said to be \emph{weakly horizontally conformal to first order at $x$} if for all non-zero eigenvalues $\lambda_1,...,\lambda_n$ of the first fundamental form $\varphi^*(h)$ one has:
\begin{equation}
\lambda_1(x) = ... = \lambda_{n}(x),\qquad \nabla \lambda_1(x) =...=\nabla \lambda_{n}(x).\label{eq: infi-hor-conf}
\end{equation}
\end{defn}

This notion allows for an infinitesimal Koszul type formula, due to Gudmundsson \cite{gud-thesis}, relating the Levi-Civita connections of $M$ and $N$. We give the formula below and for completeness we include a proof.

\begin{proposition}[\cite{gud-thesis}]\label{prop: whc-connection}
Let $(M,g)$, $(N,h)$ be Riemannian manifolds, $\varphi:M\to N$ a submersive map. Let $X,Y,Z$ be vector fields on $N$ and denote their horizontal lifts to $M$ by $\widehat X,\widehat Y,\widehat Z$. Then at all points $x\in M$ at which $\varphi$ is weakly horizontally conformal to first order one has:
\begin{equation}
h(\nabla^N_{X}Y,Z)=\lambda^2\,g(\nabla^M_{\widehat X}\widehat Y,\widehat Z)+\widehat X(\ln\lambda)\,h(Y,Z)+\widehat Y(\ln\lambda)\,h(X,Z)-\widehat Z(\ln\lambda)\,h(X,Y),
\end{equation}
where $\nabla^M,\nabla^N$ denote the Levi-Civita connections on $M,N$ respectively and $\lambda$ the conformality factor of $\varphi$ at $x$.
\end{proposition}
\begin{proof}
Suppose $\varphi$ is weakly horizontally conformal to first order at $x\in M$, meaning that the one-jet of the first fundamental form $\varphi^*(h)$ agrees with $\lambda^2\, g\circ\mathcal H$, where $\mathcal H$ denotes the orthogonal projection to the horizontal bundle (cf. Defintion \ref{def: inf-whc}). Note that $X=d\varphi(\widehat X)$, $Y=d\varphi(\widehat Y)$, $Z=d\varphi(\widehat Z)$ and then one has at $x$:
$$X(h(Y,Z))\circ\varphi =\widehat X(\varphi^*(h)\,(\widehat Y,\widehat Z))= \widehat X(\lambda^2 g(\widehat Y,\widehat Z)) = \widehat X(\lambda^2) g(\widehat Y,\widehat Z)+\lambda^2 \widehat X(g(\widehat Y,\widehat Z))$$
with similar expressions holding for other compositions. Applying the Koszul formula for $\nabla^N$ then gives:
\begin{align*}
2 h(\nabla^N_X Y,Z)=& X(h(Y,Z)) + Y(h(X,Z))- Z(h(X,Y))+h([X,Y],Z)-h([X,Z],Y)-h([Y,Z],X)\\
\begin{split}
=& \lambda^2 \left(\widehat X(g(\widehat Y,\widehat Z)) + \widehat Y(g(\widehat X,\widehat Z))-\widehat Z(g(\widehat X,\widehat Y))+g([\widehat X,\widehat Y],\widehat Z)-g([\widehat X,\widehat Z],\widehat Y])-g([\widehat Y,\widehat Z],\widehat X\right)\\
&\quad+2\widehat X(\ln\lambda) \lambda^2 g(\widehat Y,\widehat Z)+2\widehat Y(\ln\lambda)\lambda^2g(\widehat X,\widehat Z)-2\widehat Z(\ln\lambda)\lambda^2g(\widehat X,\widehat Y)
\end{split}\\
=& 2\lambda^2 g(\nabla^M_{\widehat X}\widehat Y,\widehat Z)+2\widehat X(\ln\lambda) h(Y,Z)+2\widehat Y(\ln\lambda)h(X,Z)-2\widehat Z(\ln\lambda)h(X,Y),
\end{align*}
which is the desired relation.\end{proof}

The key infinitesimal calculation for deriving the reduction statements is now the following:

\begin{lemma}\label{lemma: taup-B}
Let $(M,g)$, $(N,h)$ be Riemannian manifolds, $\varphi:M\to N$ submersive, $B\subset N$ a co-dimension $p$ submanifold. Suppose $\varphi$ is weakly horizontally conformal to first order at $x\in\varphi^{-1}(B)$, then for all $Z\in T_{\varphi(x)}N$ with $Z\perp T_{\varphi(x)}B$:
\begin{equation}
h(\frac{\tau_p(\varphi)\,(x)}{\|d\varphi\|^{p-2}},Z)=h(d\varphi(H_{\varphi^{-1}(B)}),Z)- h(\lambda^2 H_B,Z),\label{eq: taup-B}
\end{equation}
where $\lambda$ is the conformality factor of $\varphi$ at $x$ and $H_{B}, H_{\varphi^{-1}(B)}$ denote the mean curvatures of $B$, $\varphi^{-1}(B)$ at $\varphi(x)$, $x$ respectively. $\tau_p(\varphi)\,(x)$ is the $p$-tension field of $\varphi$ evaluated at $x$, recall equation~(\ref{eq: p-tension-def}).
\end{lemma}
\begin{proof}
Let $y=\varphi(x)$, we let
\begin{enumerate}[label=$\bullet$]
\item $v_1,...,v_{m-n}$ be an orthonormal basis of the fibre $T_x \varphi^{-1}(\{y\})$.
\item $b_1,...,b_{n-p}$ an orthonormal basis of $T_yB$ and $\widehat{b}_1,...\widehat{b}_{n-p}$ their horizontal lifts to $T_x \varphi^{-1}(B)$.
\item $p_1,...,p_{p}$ an orthonormal basis of $(T_yB)^\perp$ and $\widehat{p}_1,...,\widehat{p}_{p}$ their horizontal lifts to $T_xM$.
\end{enumerate}
Note that
$$v_1,...,v_{m-n},\lambda\widehat{b}_1,...,\lambda\widehat{b}_{n-p},\lambda \widehat{p}_1,....,\lambda\widehat{p}_{p}$$
is an orthonormal basis of $T_xM$. For $Z$ orthogonal to $B$ we denote its horizontal lift by $\widehat Z$. By the definition of the tension field $\tau$ (recall equation~(\ref{eq: def-tau})):
\begin{align}
h(\tau(\varphi),Z)&=h(\sum_i\nabla^{\varphi^*}_{\lambda \widehat{b}_i} \lambda b_i+\sum_i \nabla^{\varphi^*}_{\lambda \widehat{p}_i} \lambda p_i  - d\varphi(\sum_i\nabla^M_{v_i}v_i+\sum_i \nabla^M_{\lambda\widehat{b}_i}\lambda\widehat{b_i}+\sum_i \nabla^M_{\lambda\widehat{p}_i}\lambda\widehat{p_i}),Z)\notag\\
&=h(\sum_i\nabla^{\varphi^*}_{\lambda \widehat{b}_i} \lambda b_i+\sum_i \nabla^{\varphi^*}_{\lambda \widehat{p}_i} \lambda p_i,Z)-\lambda^2g(\sum_i\nabla^M_{v_i}v_i+\sum_i \nabla^M_{\lambda\widehat{b}_i}\lambda\widehat{b_i}+\sum_i \nabla^M_{\lambda\widehat{p}_i}\lambda\widehat{p_i},\widehat Z).\label{eq: lemma-tau2}
\end{align}
Here $\nabla^{\varphi^*}$ denotes the pullback connection of $\nabla^N$ to $\varphi^*(TN)$. Note that $Z\perp b_i$ so that
$$h(\sum_i\nabla^{\varphi^*}_{\lambda\widehat{b}_i}\lambda {b}_i,Z) = \lambda^2 h(\sum_i\nabla^N_{b_i}b_i,Z)=-\lambda^2 h(H_B,Z).$$
Further since $\widehat Z$ is orthogonal to $\varphi^{-1}(B)$ one has:
$$\lambda^2 g(-\sum_i \nabla^M_{v_i}v_i - \sum_i \nabla^M_{\lambda \widehat{b}_i}\lambda\widehat{b}_i,\widehat Z)=\lambda^2 g(H_{\varphi^{-1}(B)},\widehat Z) = h(d\varphi(H_{\varphi^{-1}(B)}),Z).$$
Two terms remain unevaluated on the right-hand side of equation~(\ref{eq: lemma-tau2}). For these we note first that:
$$h(\sum_i\nabla^{\varphi^*}_{\lambda\widehat p_i}\lambda p_i, Z)-\lambda^2 g(\sum_i\nabla^M_{\lambda \widehat{p}_i}\lambda\widehat p_i, \widehat Z)= \lambda^2 h(\sum_i\nabla^N_{p_i}p_i,Z) - \lambda^2\cdot\lambda^2 g(\sum_i\nabla^M_{\widehat p_i}\widehat p_i,\widehat Z).$$
In the final calculation we apply Proposition \ref{prop: whc-connection} to get:
\begin{align*}
-\lambda^4g(\sum_i\nabla^M_{\widehat p_i}\widehat p_i, Z)&= - \lambda^2h(\sum_i \nabla_{p_i}^N p_i, Z) + 2 \sum_i\lambda^2 \widehat p_i(\ln\lambda) h(p_i,Z) - \lambda^2 \widehat Z(\ln\lambda) \sum_i h(p_i,p_i)\\
&= -\lambda^2 h(\nabla^N_{p_i}p_i,Z)-(p-2) h(d\varphi(\nabla\ln\lambda),Z).
\end{align*}
Here we used that $\lambda\widehat p_i(\lambda)h(p_i,Z) = g(\nabla\ln\lambda,\lambda\widehat p_i) \lambda^2 g(\lambda\widehat p_i,\widehat Z)$, summing this over $i$ gives $\lambda^2 g(\nabla\ln\lambda,\widehat Z)$ since $\widehat Z$ is a linear combination of the $\lambda\widehat p_i$.

Combining all terms in (\ref{eq: lemma-tau2}) then gives:
$$h(\tau(\varphi)\,(x)+(p-2)d\varphi(\nabla\ln\lambda)\,(x),Z)= h(d\varphi(H_{\varphi^{-1}(B)}),Z)-h(\lambda^2H_B,Z),$$
which is the desired identity (\ref{eq: taup-B}).
\end{proof}

We now prove \cref{thm: co-dim-2}.

\begin{proof}[Proof of \cref{thm: co-dim-2}]
The implication (i)$\implies$(ii) follows immediately from equation~(\ref{eq: taup-B}) of \cref{lemma: taup-B}. Similarly (ii)$\implies$(i) follows from \cref{lemma: taup-B} and the fact that for any $y\in N$ and any subvectorspace $V\subset T_yN$ there exists a minimal submanifold $B\subset N$ with $y\in B$ and $T_yB=V$.
\end{proof}

\begin{remark}
The (purely local) statement that any tangent space of $N$ can be realised as the tangent space of a minimal submanifold is a standard fact in Riemannian geometry, although we could not find an explicit reference for this statement. It follows from the Banach space implicit function theorem, similarly to how one shows the existence of harmonic functions with prescribed derivatives at a point, see e.g.\ Appendix A.1, Lemma~A.1.1 of \cite{baird-wood-book} or the discussion surrounding Appendix K, Theorem~45 of \cite{besse-einstein}. For convenience of the reader we provide a proof in the appendix.
\end{remark}

Observe that the following reduction type theorem is an immediate consequence of equation~(\ref{eq: taup-B}) in \cref{lemma: taup-B}:

\begin{theorem}\label{thm: reduc-infi}
Let $(M,g)$, $(N,h)$ be Riemannian manifolds, $\varphi\colon M\to N$ submersive and $B\subset N$ a submanifold of co-dimension $p$. Suppose $\varphi$ is horizontally conformal to first order at all $x\in\varphi^{-1}(B)$, then any two of the following imply the third:
\begin{enumerate}[label=(\roman*)]
\item $\tau_p(\varphi)\,(x)\in T_{\varphi(x)}B$ for all $x\in\varphi^{-1}(B)$.
\item $B$ is minimal in $N$.
\item $\varphi^{-1}(B)$ is minimal in $M$.
\end{enumerate}
\end{theorem}

And now \cref{thm: co-dim-p} follows from the definition of $\tau_p$:

\begin{proof}[Proof of \cref{thm: co-dim-p}]
Recall that $\tau_p(\varphi)=\|d\varphi\|^{p-2}(\tau(\varphi)+(p-2)d\varphi(\nabla\ln\|d\varphi\|))$. If $\varphi$ is a harmonic morphism then $\tau(\varphi)=0$ and $\tau_p(\varphi)$ is proportional to $d\varphi(\nabla\|d\varphi\|^2)$ for $p\neq2$. \cref{thm: co-dim-p} then follows from \cref{thm: reduc-infi}.
\end{proof}

\section{Examples: area-minimising cones of $\mtr^m$}\label{sec: area-min}

In this section we will apply the results of \cref{sec: reduc} in order to generate interesting examples of minimal submanifolds. The harmonic morphisms we will consider for this will be homogeneous polynomials of degree $2$, classified by Ou and Wood \cite{ou-wood-96,ou-97}. The data determining these maps is essentially given by so-called \emph{Clifford systems}. The relevant class of quadratic polynomial harmonic morphisms which, in arbitrary co-dimension, pull back minimal cones to minimal cones are the \emph{umbillic} ones. Pulling back $P^{-1}(V)$ for $V\subset\mtr^n$ a subvectorspace by an umbillic quadratic harmonic morphism $P\colon\mtr^m\to\mtr^n$ recovers classical families of minimal submanifolds, namely cones over the focal varieties of FKM isoparametric foliations of $S^{m-1}$ \cite{FKM}.\medbreak

While we develop the remarks of the above paragraph in \cref{sec: umbillic}, the main part will be the construction of a novel family of area-minimising co-dimension $1$ cones in $\mtr^m$ in \cref{sec: minimal-cone}. These are built by pulling back the Simons cone
$$S_n\defeq\{ (x,y)\in\mtr^{2n}\mid \|x\|^2=\|y\|^2\}$$
by an umbilic quadratic harmonic morphism $P\colon\mtr^m\to\mtr^{2n}$. We show that if $m\geq32$ and $2n\geq8$ these singular cones are area-minimisers, cf.\ \cref{thm: minimal-Scone}.

\subsection{Umbillic polynomial harmonic morphisms}\label{sec: umbillic}

\begin{defn}
We say that a homogeneous polynomial harmonic morphism $P\colon\mtr^m\to\mtr^n$ is \emph{umbillic} if there exists a function $\eta\colon\mtr^m\to\mtr$ so that
$$dP(\nabla \|dP\|^2\,(x))=\eta(x)\,P(x)$$
for all $x\in\mtr^m$.
\end{defn}

An immediate application of \cref{thm: co-dim-p}, and our motivation for this definition, is that an umbillic harmonic morphism pulls back minimal cones to minimal cones:
\begin{corollary}\label{cor: umbillic}
Let $P\colon\mtr^m\to\mtr^n$ be an umbillic homogeneous polynomial harmonic morphism.
\begin{enumerate}
\item For any minimal conical submanifold $C\subset\mtr^n\setminus\{0\}$ the pre-image $P^{-1}(C)\subset\mtr^m\setminus\{0\}$ is a minimal conical submanifold.
\item For any minimal submanifold $B\subset S^{n-1}$ of $S^{n-1}$ the pullback $P^{-1}( \mtr_{>0}\cdot B)\cap S^{m-1}$ is a minimal submanifold of $S^{m-1}$.
\end{enumerate}
\end{corollary}
\begin{proof}
Point 1 follows from \cref{thm: co-dim-p} by noting that $dP(\nabla\|P\|^2(x))=\eta(x)\cdot P(x)$ is proportional to the radial vector, which is always tangent to any cone. Point 2 follows from noting that a submanifold $B\subset S^{n-1}$ is minimal in $S^{n-1}$ if and only if the cone $\mtr_{>0}\cdot B$ is minimal in $\mtr^n$. 
\end{proof}
\begin{remark}
The submanifold $P^{-1}( \mtr_{>0}\cdot B)\cap S^{m-1}$ of point 2. in \cref{cor: umbillic} is not necessarily closed, even if $B$ is closed.
\end{remark}

We wish to apply \cref{cor: umbillic} in order to generate interesting examples of minimal cones. In what follows we will briefly describe the classification of umbillic polynomial harmonic morphisms of degree $2$. For completeness we include a short proof.

\begin{defn}\label{def: clifford-system}
A tuple $(A_1,...,A_n)$ of symmetric real $(m\times m)$ matrices is called a \emph{Clifford system on $\mtr^m$} if
$$A_iA_j+A_jA_i=2\delta_{ij}\mathds1_{m\times m}$$
for all $i,j\in\{1,...,n\}$.
\end{defn}
\begin{remark}
Clifford systems have been classified \cite{atiyah-bott-shapiro-64,FKM}. For a given $n$ each Clifford system on $\mtr^n$ is a direct sum $(\bigoplus_\ell A_1^{(\ell)},...,\bigoplus_\ell A_n^{(\ell)})$ of \emph{irreducible} Clifford systems. If $n\not\equiv 1$ mod $4$ then there exists exactly one irreducible Clifford system (up to conjugation with an isometry of $\mtr^m$). If $n\equiv1$ mod $4$ there are exactly two, and they have the same dimension.
\end{remark}

In \cref{table: clifford} below we list the dimensions $\underline{m}(n)$ of the irreducible Clifford systems on $\mtr^n$.
\begin{table}[!h]
\begin{center}
\begin{tabular}{c | c | c | c | c | c | c| c | c | c | c}
$n$ & 1 & 2 & 3 & 4 & 5 & 6 & 7 & 8 & $\cdots$ & $n+8$\\
\hline
$\underline{m}(n)$ & 1 & 2 & 4 & 8 & 8 & 16& 16 & 16& $\cdots$ &$16\,\underline{m}(n)$
\end{tabular}
\caption{Dimensions of irreducible Clifford systems. In the literature the usual convention is to tabulate instead the related quantity $\delta(n)=\frac12 \underline{m}(n+1)$.}\label{table: clifford}
\end{center}
\end{table}

\begin{theorem}[\cite{ou-wood-96,ou-97}]\label{thm: deg2}
Let $P:\mtr^m\to\mtr^n$ be a homogeneous polynomial harmonic morphism of degree $2$. Then $P$ is a weighted sum of Clifford systems. That is up to an isometry of the domain:
$$\mtr^m=\bigoplus_{\ell=1}^k \mtr^{m_\ell}\oplus\mtr^{m_0},\qquad P((x_1,...,x_k),x_0)=\sum_\ell\lambda_\ell \left(\langle x_\ell, A^{(\ell)}_1 x_\ell\rangle ,...,\langle x_\ell, A^{(\ell)}_nx_\ell\rangle\right)$$
for Clifford systems $(A_1^{(\ell)},...,A_n^{(\ell)})$ on $\mtr^{m_\ell}$ and $\lambda_\ell>0$. $P$ is umbillic if and only if all weights $\lambda_\ell$ agree.
\end{theorem}
\begin{proof}
Since $P$ is a homogeneous quadratic polynomial there are symmetric matrices $A_1,...,A_n\in M_{m\times m}(\mtr)$ so that
$$P(x)=(\langle x,A_1x\rangle,...,\langle x,A_nx\rangle).$$
If $P_i$ denotes the $i$-th component of $P$, then the condition that $P$ is weakly horizontally conformal means $\langle \nabla P_i, \nabla P_j\rangle = \delta_{ij}\|\nabla_iP\|^2$, recall \cref{lemma: whc-euc}. Since the $A_i$ are symmetric, this means
$$\langle x, \frac12(A_iA_j+A_jA_i)x\rangle=\langle A_ix,A_jx\rangle=\delta_{ij}\langle x, A_i^2x\rangle$$
for all $x\in\mtr^m$, i.e.
\begin{equation}
A_iA_j+A_jA_i=2\delta_{ij}\,A_i^2\label{eq: hom2-sq}
\end{equation}
for all $i,j\in\{1,...,n\}$. Let $\lambda_1^2,...,\lambda_k^2$ be the non-zero eigenvalues of $A_i^2$ and $E_1,..,E_k$ the associated eigenspaces. Note that this decomposition does not depend on $i$ since $A_i^2=A_j^2$ for all $i,j$. Then each $A_i$ decomposes as a direct sum $A_i=\bigoplus_{\ell}\lambda_\ell A_i^{(\ell)}$, where $A_i^{(\ell)}$ are symmetric endomorphisms of $E_\ell$. Equation~(\ref{eq: hom2-sq}) then becomes
$$A_i^{(\ell)}A_j^{(\ell)}+A_j^{(\ell)}A_{i}^{(\ell)}=2\delta_{ij}\,\mathrm{id}_{E_\ell},$$
from which the first part of the theorem follows. The second part follows from
$$dP_i(\nabla \|\nabla P_i\|^2)\,(x)=8\sum_\ell \lambda_{\ell} \langle x, (A_i^{(\ell)})^3x\rangle= 8\sum_\ell \lambda_{\ell}^3 \langle x, A_i^{(\ell)}x\rangle.$$
Taking the quotient by $P_i(x)= \sum_\ell \lambda_{\ell} \langle x, A_i^{(\ell)}x\rangle$ gives a function independent of $i$ if and only if the $\lambda_\ell$ all agree.
\end{proof}

\begin{example}Consider the $4\times 4$ matrices:
$$
B_1=\begin{pmatrix}0&-1&0&0\\1&0&0&0\\0&0&0&1\\0&0&-1&0\end{pmatrix},\quad 
B_2=\begin{pmatrix}0&0&-1&0\\0&0&0&1\\1&0&0&0\\0&-1&0&0\end{pmatrix},\quad
B_3=\begin{pmatrix}0&0&0&-1\\0&0&-1&0\\0&1&0&0\\1&0&0&0\end{pmatrix}.$$
\end{example}
Then the $8\times8$ matrices
$$
\begin{pmatrix}0&\mathds1_{4\times4}\\ \mathds1_{4\times 4}&0\end{pmatrix},\quad
\begin{pmatrix}0&B_1\\ -B_1&0\end{pmatrix},\quad
\begin{pmatrix}0&B_2\\ -B_2&0\end{pmatrix},\quad
\begin{pmatrix}0&B_3\\ -B_3&0\end{pmatrix},$$
form a Clifford system on $\mtr^8$. The associated quadratic harmonic morphism $\mtr^8\to\mtr^4$ is, up to an isometric identification of domain and co-domain, the multiplication $\mth\oplus\mth\to\mth$, $(x,y)\mapsto 2x\cdot_\mth y$ of two quaternions.

\subsection{Minimal submanifolds from quadratic harmonic morphisms}\label{sec: minimal-cone}
We briefly describe how pulling back subvectorspaces $V\subset\mtr^n$ by umbillic quadratic harmonic morphisms recovers classical minimal cones. For simplicity we assume the harmonic morphism is full:
\begin{defn}
A map $f:\mtr^m\to\mtr^n$ is called \emph{full} if
$$\spn\{(\nabla \langle f, v\rangle)\,(x)\mid x\in\mtr^m,v\in\mtr^n\}=\mtr^m,$$
i.e.\ $f$ does not factor over any subvectorspace of $\mtr^m$.
\end{defn}

\begin{proposition}
Let $V\subset\mtr^n$ be a subvectorspace of dimension $k$ and $P:\mtr^m\to\mtr^n$ a full umbillic quadratic harmonic morphism. Then $P^{-1}(V)\cap S^{m-1}$ is a closed minimal submanifold of $S^{m-1}$ of co-dimension $n-k$.
\end{proposition}
\begin{proof}
$P^{-1}(V)\setminus\{ x\in\mtr^m\mid \|dP\|^2(x)=0\}$ is a minimal submanifold by \cref{thm: co-dim-p}. Since $P$ is full one has $\|dP\|^2(x)=0$ if and only if $x=0$ from \cref{thm: deg2} (recall that the $A_i$ are invertible). Then $P^{-1}(V)\setminus\{0\}$ is a minimal conical submanifold with closure $P^{-1}(V)$. It follows that $P^{-1}(V)\cap S^{m-1}$ is a closed minimal submanifold.
\end{proof}

\begin{remark}Note that if $n\geq2$ then $m$ is always even by \cref{table: clifford}. We assume this to be the case, then:
\begin{enumerate}
\item If $\dim(V)=m-1$ then $P^{-1}(V)$ is congruent to the Simons cone $S_{m/2}=\{(x,y)\in\mtr^{m/2}\oplus\mtr^{m/2}\mid \|x\|^2=\|y\|^2\}$. If $m\geq8$ then $S_{m/2}$ is an area-minimising cone in $\mtr^m$ \cite{bgg-69}.
\item If $\dim(V)=m-2$ then $P^{-1}(V)$ is congruent to the complex quadric, that is the cone $\mtr_{>0}\cdot V_2(\mtr^{m/2})$ over the Stiefel-manifold of two frames $V_2(\mtr^{m/2})=\{(x,y)\in S^{m/2-1}\times S^{m/2-1}\mid \langle x,y\rangle = 0\}$. This is always an area-minimising cone in $\mtr^m$, since it is calibrated by a power of the K\"ahler form. See also \cite{jxc-24} for a proof using Lawlor's curvature criterion \cite{lawlor-91}.
\item More generally $P^{-1}(V)$ is the cone $\mtr_{>0}\cdot M_+$ over a focal manifold of an isoparametric foliation of $S^{m-1}$ of FKM-type \cite{FKM}.
\end{enumerate}
\end{remark}

We now give the main theorem of this section.

\begin{theorem}\label{thm: minimal-Scone}
Let $m\geq32$ and $n\geq4$. For $A_1,...,A_{2n}\in M_{m\times m}(\mtr)$ a Clifford system the cone
$$C^4_{m,n}\defeq\left\{x\in\mtr^m\ \middle|\ \sum_{i=1}^n \langle x,A_ix\rangle^2=\sum_{i=1}^n \langle x,A_{n+i}x\rangle^2\right\}$$
is an area-minimising hypersurface in $\mtr^m$.
\end{theorem}

\begin{remark}\label{rem: P-cone}
The cone $C^4_{m,n}$ defined in \cref{thm: minimal-Scone} is equal to $P^{-1}(S_n)$, the pullback of the Simons cone $S_n=\{(x,y)\in\mtr^n\oplus\mtr^n\mid \|x\|^2=\|y\|^2\}$ by a full umbillic quadratic harmonic morphism $P\colon\mtr^m\to\mtr^{2n}$. As such the mean curvature of $C^4_{m,n}$ vanishes at all points of $C^4_{m,n}\setminus P^{-1}(\{0\})$, regardless of the values of $(m,n)$.
\end{remark}
\begin{remark}\label{rem: area-min}
\begin{enumerate}
\item The singular set of $P^{-1}(S_n)$ is given by $P^{-1}(\{0\})$, which is $m-2n$ dimensional. Since the singular set of an area-minimising hypersurface of $\mtr^m$ is at most of dimension $m-8$ \cite{federer-70}, the condition $n\geq4$ in \cref{thm: minimal-Scone} is clearly a necessary condition for $P^{-1}(S_n)$ to be an area-minimiser.
\item There exists a Clifford system $A_1,...,A_{2n}\in M_{m\times m}(\mtr)$ if and only if $m=k\,\underline{m}(2n)$ for some $k\in\mtn$, here $\underline{m}(2n)$ is the number from \cref{table: clifford}. Since $\underline{m}(2\cdot4)=16$ and $\underline{m}(2n)\geq32$ if $n\geq5$, it follows that \cref{thm: minimal-Scone} exhausts all possibilities except $(m,n)=(16,4)$.
\end{enumerate}\end{remark}

The simplest way to prove \cref{thm: minimal-Scone} is to use the method of subcalibrations. This method was developed by De Philippis and Paolini \cite{philippis-paolini-09} to provide a simple proof that the Simons cone is area-minimising. For convenience of the reader we briefly review the relevant definitions and theorems in \cref{sec: subcalib} below. The proof then follows in \cref{sec: Scone-proof}.

\subsection{The method of subcalibrations}\label{sec: subcalib}

\begin{defn}\label{def: pp-09} Let $\Omega\subseteq\mtr^m$ be open and $E\subset\Omega$ be open.
\begin{enumerate}
\item The \emph{perimeter} of $E$ in $\Omega$ is
$$P(E,\Omega)=\sup\left\{\int_E \mathrm{div}(g)\ \middle|\ g\in C_c^1(\Omega\to\mtr^m), \sup_x\|g(x)\|\leq1\right\}$$
\item $E$ is \emph{minimal} if for all $A\subseteq\Omega$ bounded and open one has
$$P(E\cap A, \Omega\cap A)\leq P(F\cap A, \Omega\cap A)\quad\text{whenever $(F\cup E)\setminus (F\cap E)\Subset A$}.$$
\item $E$ is \emph{subminimal} if for all $A\subseteq\Omega$ bounded and open one has
$$P(E\cap A, \Omega\cap A)\leq P(F\cap A, \Omega\cap A)\quad\text{whenever $F\subseteq E$, $E\setminus F\Subset A$}.$$
\item Suppose $\partial E$ has $C^2$ regularity. A vectorfield $\xi:\Omega\to\mtr^m$ is called a \emph{subcalibration} of $E$ if
\begin{enumerate}[label=(\roman*)]
\item $\xi(x)$ is equal to the exterior normal of $\partial E$ for all $x\in\partial E$.
\item $\mathrm{div}(\xi)\,(x)\leq0$ for all $x\in E$.
\item $\|\xi(x)\|\leq1$ for all $x\in \Omega$.
\end{enumerate}
\end{enumerate}
\end{defn}

\begin{remark}
$E$ is subminimal if and only if $\partial E$ is area-minimising with respect to variations that go into the set $E$, i.e.\ inward variations.
\end{remark}

\begin{theorem}[\cite{philippis-paolini-09}]\label{thm: dpp09}Let $E\subset\Omega\subseteq\mtr^m$ be open.\begin{enumerate}
\item $E$ is minimal in the sense of Definition~\ref{def: pp-09} if and only if $\partial E$ is an area-minimiser in $\Omega$.
\item If $E$ and $\Omega\setminus\overline E$ are subminimal, then $E$ is minimal.
\item Suppose $\partial E$ has $C^2$ regularity and $E$ admits a subcalibration, then $E$ is subminimal.
\item Let $E_k\subset E$ be a sequence with $E_k\overset{L^1\subs{loc}}\longrightarrow E$, meaning for all bounded $A$ the measure of $A\cap(E\setminus E_k)$ converges to $0$. If $E_k$ are all subminimal then $E$ is subminimal.
\end{enumerate}
\end{theorem}
\subsection{Proof of \cref{thm: minimal-Scone}}\label{sec: Scone-proof}

We proceed as follows, let $s\colon\mtr^{2n}\to\mtr$, $(x,y)\mapsto\|x\|^2-\|y\|^2$ so that $S_n=s^{-1}(\{0\})$ is the Simons cone. There is an umbillic quadratric harmonic morphism $P\colon\mtr^m\to\mtr^{2n}$ so that $C^4_{m,n}=(s\circ P)^{-1}(\{0\})$, as explained in \cref{rem: P-cone}. We now consider the function
$$f\colon\mtr^m\to\mtr^n,\qquad x\mapsto \|P(x)\|^{1/2}\,\|x\|^2\,(s\circ P)(x).$$
The first step is to show:
\begin{lemma}\label{lemma: subcal}
Let $m\geq32$, $n\geq4$. Then $\frac1{\|\nabla f\|}\nabla f$ is a sub-calibration of
$$\{ f<-\epsilon\}\coloneq\{x\in\mtr^m\mid f(x)<-\epsilon\}$$
in $\Omega=\mtr^m\setminus P^{-1}(\{0\})$.
\end{lemma}
\begin{proof}
Clearly $\frac{\nabla f}{\|\nabla f\|}$ is equal to the exterior normal on $\partial\{f<-\epsilon\}=\{f=-\epsilon\}$ and $\|\frac{\nabla f}{\|\nabla f\|}\|\leq1$ everywhere. It remains to check that $\mathrm{div}(\frac1{\|\nabla f\|}\nabla f)\leq0$ on $\{f<-\epsilon\}$. To compute this we note first
\begin{equation}
\mathrm{div}(\frac1{\|\nabla f\|}\nabla f)= \frac{\|\nabla f\|^2\,\Delta(f)-\frac12\langle\nabla\|\nabla f\|^2,\nabla f\rangle}{\|\nabla f\|^3}.\label{eq: div-f}
\end{equation}

Recall that $P(x)=(\langle x,A_1x\rangle ,....,\langle x,A_{2n}x\rangle)$ for some Clifford system $A_1,...,A_{2n}$. We now claim:
\begin{align*}
&\|\nabla(s\circ P)\,(x)\|^2=16\|P(x)\|^2\,\|x\|^2,& &\left\|\nabla \|P\|^2(x)\right\|^2 = 16\|P(x)\|^2\,\|x\|^2, \\
&\langle \nabla(s\circ P)\,(x),\nabla\|P\|^2\,(x)\rangle= 16\|x\|^2(s\circ P)(x),& &\|\nabla\|x\|^2\,\|^2=4\|x\|^2,\\
&\langle \nabla(s\circ P)\,(x),\nabla\|x\|^2\rangle=8(s\circ P)(x),&  &\langle \nabla\|P\|^2\,(x),\nabla\|x\|^2\rangle =8\|P(x)\|^2,\\
&\Delta(s\circ P) \,(x)= 0,& &\Delta(\|P\|^2)\,(x)=8\|x\|^2\cdot 2n,\\
& \Delta(\|x\|^2)=2\,m.
\end{align*}
\textbf{Proof of claim.} First we remark that all relations involving scalar products with $\nabla\|x\|^2$ are obvious since $dF(\nabla\|x\|^2)=2\deg(F)\,F$ for any homogeneous function $F$.

Second one has:
\begin{align*}
\nabla\|P\|^2(x) &= \nabla\sum_{i=1}^{2n} P_i(x)^2 = 2\sum_{i=1}^{2n} P_i(x)\nabla P_i(x),\\
\nabla(s\circ P)&=\nabla\sum_{i=1}^n P_i(x)^2-\sum_{i=1}^nP_{i+n}(x)^2=2\sum_{i=1}^n P_i(x)\nabla P_i(x) - 2\sum_{i=1}^nP_{i+n}(x)\nabla P_{i+n}(x).
\end{align*}
recalling
$$\langle \nabla P_i(x),\nabla P_j(x)\rangle =4\langle x, A_i A_j x\rangle=2\langle x, (A_iA_j+A_jA_i)x\rangle = 4\|x\|^2\delta_{ij}$$
the remaining scalar products become easy calculations as well.

Since $s$ is harmonic and $P$ is a harmonic morphism $\Delta(s\circ P)=0$ is immediate. Additionally:
$$\Delta(\|P\|^2)(x)=\mathrm{div}\left(2\sum_{i=1}^{2n}P_i(x) \nabla P_i(x)\right) = 2\sum_{i=1}^{2n}\|\nabla P_i(x)\|^2+\sum_i P_i(x)\Delta(P_i)(x)=8\cdot (2n)\cdot \|x\|^2,$$
and the last remaining equation follows. This proves the claim.

With these computational rules one may calculate, either by hand or with a software, from equation~(\ref{eq: div-f}) that $\mathrm{div}(\frac1{\|\nabla f\|}\nabla f) $ is equal to
\begin{align*}
{\bigg(}(12&n-47)\|x\|^{10}\|P\|^{\frac92}+(\frac{33+198n}8)\|x\|^{10}\|P\|^{\frac 12}(s\circ P)^2+ (2m-41)\|x\|^6\|P\|^{\frac{13}2}\\
&+(\frac{384n+66m-681}{16})\|x\|^6\|P\|^{\frac52}(s\circ P)^2+(10+4m)\|x\|^2\|P\|^{\frac92}(s\circ P)^2{\bigg)}\cdot \frac{(s\circ P)}{\|\nabla f\|^3}.
\end{align*}
For $m\geq32$ and $n\geq4$ every summand in the large parenthesis is positive. On the other hand $(s\circ P)$ has the same sign as $f$ and so is negative on $\{f <-\epsilon\}$. This completes the proof of the proposition.
\end{proof}

\begin{lemma}\label{lemma: subminimal}
$\{f<-\epsilon\}$ is subminimal in $\mtr^m$.
\end{lemma}
\begin{proof}
By \cref{lemma: subcal} $\{f<-\epsilon\}$ is subcalibrated, and hence subminimal, in $\mtr^m\setminus P^{-1}(\{0\})$. Since $P^{-1}(\{0\})$ is a rectifiable set of co-dimension $2n\geq 2$ it follows for any open sets $E\subset\Omega$ that
$$P(E,\Omega)=P(E\setminus P^{-1}(\{0\}), \Omega\setminus P^{-1}(\{0\})),$$
and hence $E$ is subminimal in $\Omega$ if and only if $E\setminus P^{-1}(\{0\})$ is subminimal in $\Omega\setminus P^{-1}(\{0\})$. 
\end{proof}

\begin{lemma}\label{lemma: L1}
As $\epsilon\to0$ one has $\{f<-\epsilon\}\overset{L^1\subs{loc}}\longrightarrow \{f<0\}$.
\end{lemma}
\begin{proof}
The sets $\{f<-\epsilon\}$ are monotone in $\epsilon$ and $\bigcup_\epsilon\{f<-\epsilon\}=\{f<0\}$. For any $A$ of finite volume one then has that $A\cap (\{f<0\}\setminus\{f<-\epsilon\})$ converges to $0$ by monotone continuity of measures.
\end{proof}

We now complete the proof:

\begin{proof}[Proof of \cref{thm: minimal-Scone}]
Combining \cref{lemma: subminimal,lemma: L1} we find that $\{f<0\}$ is subminimal. Just as in \cref{lemma: subcal} one proves that $\mathrm{div}(\frac{-\nabla f}{\|\nabla f\|})$ is a subcalibration of $\{f>\epsilon\}$, repeating the argument of \cref{lemma: subminimal,lemma: L1} then shows that $\{f>0\}$ is also subminimal. \cref{thm: minimal-Scone} then follows from point 2. of \cref{thm: dpp09}.
\end{proof}
\appendix
\section{Local existence of minimal submanifolds}\label{app: exist}

\begin{theorem}\label{thm: appendix}
Let $(M,g)$ be a Riemannian manifold, $p\in M$, and $V\subset T_xM$ a subvectorspace. Then there is a minimal submanifold $B\subset M$ with $x\in M$ and $T_xB=V$.
\end{theorem}

The proof is essentially identical to a proof given by Fuglede (Lemma~A.1.1 of the Appendix of \cite{baird-wood-book}), showing the existence of harmonic functions with prescribed derivatives any point in a Riemannian manifold. While the mean curvature equation is non-linear and significantly more complicated than the Laplacian, both have the same linearisation. Local solutions can then be found by a zooming-in procedure and the Banach space implicit function theorem.

We take a coordinate chart of the form
$$B_1(0)_{m-k}\times\mtr^k\subset \mtr^{m-k}\times\mtr^k=\mtr^k$$
so that $p$ corresponds to the point $(0,0)$ and $V$ to the plane $\mtr^{m-k}\times\{0\}$. Without loss of generality one may assume that the coordinate frame is synchronous at $(0,0)$, i.e.\ the metric is of the form $g=g\subs{euc}+O(\|x\|^2)$.

\begin{defn}
\begin{enumerate}
\item Let $\alpha\in(0,1)$, $\ell\geq2$. We define
\begin{gather*}
E_1=\{ F\colon B_1(0)_{m-k}\to\mtr^k \in C^{\ell,\alpha}\mid F(0)=0, dF(0)=0, d^2F(0)=0\},\\
E_2=\{F\colon B_1(0)_{m-k}\to\mtr^k\in C^{\ell-2,\alpha}\mid F(0)=0\}.
\end{gather*}
\item For $\epsilon\in(-1,1)$, $x\in B_1(0)_{m-k}\times\mtr^k$ we let $g_\epsilon(x)\defeq g(\epsilon x)$.
\item For $F\in E_1$ we let $\Gamma(F)$ denote the graph $\{(x,F(x))\mid x\in B_1(0)_{m-k}\}$. For $\epsilon\in(-1,1)$ we let $H(\epsilon, F)\in E_2$ denote $m-k+1$,...,$m$ components of the mean curvature vector of $\Gamma(F)$ with respect to $g_\epsilon$.
\end{enumerate}
\end{defn}
\begin{remark}
\begin{enumerate}
\item Since $F(0)=0$, $dF(0)=0$ for all $F\in E_1$ one has $(0,0)\in\Gamma(F)$ and $T_{(0,0)}\Gamma(F)=\mtr^{m-k}\times\{0\}$. The condition $d^2F(0)=0$ further implies that the second fundamental form of $\Gamma(F)$ vanishes at $(0,0)$.
\item $E_i$ are Banach spaces when equipped with the $C^{\ell,\alpha}$ resp. $C^{\ell-2,\alpha}$ norms coming from the euclidean metric of $B_1(0)_{m-k}$. 
\item For $\epsilon\neq0$ one has that $(B_1(0)_{m-k}\times\mtr^k,g_\epsilon)$ is isometric to $(B_{|\epsilon|}(0)_{m-k}\times\mtr^k,\frac{g}{\epsilon^2})$, and thus has the same minimal submanifolds as $g\lvert_{B_{|\epsilon|}(0)_{m-k}\times\mtr^k}$. $g_0$ is the euclidean metric.
\item $H\colon(-1,1)\times E_1\to E_2$ is continuously differentiable, since $H(\epsilon, F)\,(x)$ is given by a (complicated) expression depending smoothly on the arguments $F(x)$, $dF(x)$, $d^2F(x)$, and $g_\epsilon(x, F(x)) = g(\epsilon x, \epsilon F(x))$, $dg_\epsilon$, $d^2g_\epsilon$.
\end{enumerate}
\end{remark}

Note that
\begin{enumerate}
\item $H(\epsilon,F)=0$ if and only if the graph $\Gamma(F)$ is minimal with respect to $g_\epsilon$.
\item $H(0,\mathbf0)=0$ for the constant function $\mathbf0$.
\item More generally if $H(\epsilon,F)^\alpha$ denotes the $\alpha$-component of $H(\epsilon,F)$ for $\alpha\in\{1,..,k\}$ one has
$$H(0,F)^\alpha=\sum_{ij=1}^{m-k} L^{ij}\partial_{ij} F^\alpha-\sum_{abij=1}^{m-k}\sum_{\beta=1}^kL^{ij} \partial_{ij}F^{\beta} F^\beta_{a} L^{ab}F^\alpha_b,$$
here $L=(\mathds1_{(m-k)\times(m-k)}+dF^T \cdot dF)^{-1}$ is the first fundamental form of $\Gamma(F)$. The complicated second summand arises from the fact that the mean curvature arises from taking the projection of $\sum_{ij}L^{ij}\partial_{ij}F^\alpha\, e_\alpha$ to the normal bundle of $\Gamma(F)$, here $e_\alpha$ denotes the standard basis of $\mtr^k$.
\item Denoting with $D_2H(\epsilon,F)\,K$ the Gateux-derivative of $H$ at the point $(\epsilon,F)$ in the direction $K\in E_1$, one then finds:
$$D_2H(0,\mathbf0)\,K=\sum_{i=1}^{m-k}\partial_{ii} K.$$
\end{enumerate}

Using the Green's function of the Laplacian of the unit ball $B_1(0)_{m-k}$ one can construct a right-inverse for the Laplacian $E_1\to E_2$, $K\mapsto \sum_{i=1}^{m-k}\partial_{ii}K$. See e.g.\ the proof of Lemma~A.1.1 in \cite{baird-wood-book} for an explicit formula. This is where the spaces $C^{\ell,\alpha}$, rather than just $C^\ell$, are required.

$E_1$ then splits as a direct sum of the kernel of the Laplacian and the image of the right-inverse, which is also a Banach space. The Banach space implicit function theorem shows the equation $H(\epsilon, \cdot)=0$ admits solutions for $\epsilon$ in some neighbourhood of $0$. These solutions are then minimal $C^{\ell,\alpha}$-submanifolds of $(B_{|\epsilon|}(0)_{m-k}\times\mtr^k, g)\subset (M,g)$. Since the Riemannian metric is smooth these manifolds are also smooth by elliptic regularity, cf.\ Theorem 4.4 of Chapter 14 in \cite{taylor-3}. This completes the proof of \cref{thm: appendix}.\qed

\bibliographystyle{alpha}
\bibliography{../my}

\providecommand{\MR}[1]{}
\begin{thebibliography}{BDGG69}

\bibitem[ABB99]{ABB-99}
Rachel Ababou, Paul Baird, and Jean Brossard.
\newblock Polyn\^{o}mes semi-conformes et morphismes harmoniques.
\newblock {\em Math. Z.}, 231(3):589--604, 1999.

\bibitem[ABBR15]{abbr-15}
Omid Amini, Matthew Baker, Erwan Brugall\'e, and Joseph Rabinoff.
\newblock Lifting harmonic morphisms {I}: metrized complexes and {B}erkovich
  skeleta.
\newblock {\em Res. Math. Sci.}, 2:Art. 7, 67, 2015.

\bibitem[ABS64]{atiyah-bott-shapiro-64}
M.~F. Atiyah, R.~Bott, and A.~Shapiro.
\newblock Clifford modules.
\newblock {\em Topology}, 3(suppl):3--38, 1964.

\bibitem[Bai83]{baird-83}
Paul Baird.
\newblock {\em Harmonic maps with symmetry, harmonic morphisms and deformations
  of metrics}, volume~87 of {\em Research Notes in Mathematics}.
\newblock Pitman (Advanced Publishing Program), Boston, MA, 1983.

\bibitem[BCD79]{bcd-79}
Alain Bernard, Eddy~A. Campbell, and Alexander~M. Davie.
\newblock Brownian motion and generalized analytic and inner functions.
\newblock {\em Ann. Inst. Fourier (Grenoble)}, 29(1):xvi, 207--228, 1979.

\bibitem[BDGG69]{bgg-69}
E.~Bombieri, E.~De~Giorgi, and E.~Giusti.
\newblock Minimal cones and the {B}ernstein problem.
\newblock {\em Invent. Math.}, 7:243--268, 1969.

\bibitem[BE81]{baird-eells-80}
Paul Baird and James Eells.
\newblock A conservation law for harmonic maps.
\newblock In {\em Geometry {S}ymposium, {U}trecht 1980 ({U}trecht, 1980)},
  volume 894 of {\em Lecture Notes in Math.}, pages 1--25. Springer, Berlin-New
  York, 1981.

\bibitem[Bes08]{besse-einstein}
Arthur~L. Besse.
\newblock {\em Einstein manifolds}.
\newblock Classics in Mathematics. Springer-Verlag, Berlin, 2008.
\newblock Reprint of the 1987 edition.

\bibitem[BG92]{baird-gudmundsson-92}
Paul Baird and Sigmundur Gudmundsson.
\newblock {$p$}-harmonic maps and minimal submanifolds.
\newblock {\em Math. Ann.}, 294(4):611--624, 1992.

\bibitem[BW03]{baird-wood-book}
Paul Baird and John~C. Wood.
\newblock {\em {H}armonic {M}orphisms between {R}iemannian {M}anifolds},
  volume~29 of {\em London Mathematical Society Monographs. New Series}.
\newblock The Clarendon Press, Oxford University Press, Oxford, 2003.

\bibitem[CC65]{const-corn-65}
Corneliu Constantinescu and Aurel Cornea.
\newblock Compactifications of harmonic spaces.
\newblock {\em Nagoya Math. J.}, 25:1--57, 1965.

\bibitem[CFO90]{cfb-90}
Laszlo Csink, P.~J. Fitzsimmons, and Bernt \O{}ksendal.
\newblock A stochastic characterization of harmonic morphisms.
\newblock {\em Math. Ann.}, 287(1):1--18, 1990.

\bibitem[Cui26]{cui-26}
Hongbin Cui.
\newblock On {FKM} isoparametric hypersurfaces in
  $\mathbb{S}^n\times\mathbb{S}^n$ and new area-minimizing cones.
\newblock Preprint, {arXiv}:2510.14650 [math.{DG}] (2026), 2026.

\bibitem[DPP09]{philippis-paolini-09}
Guido De~Philippis and Emanuele Paolini.
\newblock A short proof of the minimality of {S}imons cone.
\newblock {\em Rend. Semin. Mat. Univ. Padova}, 121:233--241, 2009.

\bibitem[EF01]{eells-fuglede-01}
J.~Eells and B.~Fuglede.
\newblock {\em Harmonic maps between {R}iemannian polyhedra}, volume 142 of
  {\em Cambridge Tracts in Mathematics}.
\newblock Cambridge University Press, Cambridge, 2001.
\newblock With a preface by M. Gromov.

\bibitem[EY95]{eells-yiu-95}
James Eells and Paul Yiu.
\newblock Polynomial harmonic morphisms between {E}uclidean spheres.
\newblock {\em Proc. Amer. Math. Soc.}, 123(9):2921--2925, 1995.

\bibitem[Fed70]{federer-70}
Herbert Federer.
\newblock The singular sets of area minimizing rectifiable currents with
  codimension one and of area minimizing flat chains modulo two with arbitrary
  codimension.
\newblock {\em Bull. Amer. Math. Soc.}, 76:767--771, 1970.

\bibitem[FKM81]{FKM}
Dirk Ferus, Hermann Karcher, and Hans~Friedrich M\"{u}nzner.
\newblock Cliffordalgebren und neue isoparametrische {H}yperfl\"{a}chen.
\newblock {\em Math. Z.}, 177(4):479--502, 1981.

\bibitem[Fug78]{fuglede-78}
Bent Fuglede.
\newblock Harmonic morphisms between {R}iemannian manifolds.
\newblock {\em Ann. Inst. Fourier (Grenoble)}, 28(2):vi, 107--144, 1978.

\bibitem[Gud]{hm-bib}
Sigmundur Gudmundsson.
\newblock The {B}ibliography of {H}armonic {M}orphisms.
\newblock
  \url{https://www.maths.lth.se/matematiklu/personal/sigma/harmonic/bibliography.html}.

\bibitem[Gud92]{gud-thesis}
Sigmundur Gudmundsson.
\newblock {\em The geometry of harmonic morphisms}.
\newblock PhD thesis, University of Leeds (Department of Pure Mathematics),
  1992.

\bibitem[HL71]{hsiang-lawson-71}
Wu-yi Hsiang and H.~Blaine Lawson, Jr.
\newblock Minimal submanifolds of low cohomogeneity.
\newblock {\em J. Differential Geometry}, 5:1--38, 1971.

\bibitem[Ish79]{ishihara-79}
T\^{o}ru Ishihara.
\newblock A mapping of {R}iemannian manifolds which preserves harmonic
  functions.
\newblock {\em J. Math. Kyoto Univ.}, 19(2):215--229, 1979.

\bibitem[JXC24]{jxc-24}
Xiaoxiang Jiao, Jialin Xin, and Hongbin Cui.
\newblock Area-minimizing cones over {S}tiefel manifolds.
\newblock {\em Adv. Math. (China)}, 53(5):929--952, 2024.

\bibitem[Law91]{lawlor-91}
Gary~R. Lawlor.
\newblock A sufficient criterion for a cone to be area-minimizing.
\newblock {\em Mem. Amer. Math. Soc.}, 91(446):vi+111, 1991.

\bibitem[LP06]{loubeau-pantilie-06}
Eric Loubeau and Radu Pantilie.
\newblock Harmonic morphisms between {W}eyl spaces and twistorial maps.
\newblock {\em Comm. Anal. Geom.}, 14(5):847--881, 2006.

\bibitem[LP10]{loubeau-pantilie-10}
Eric Loubeau and Radu Pantilie.
\newblock Harmonic morphisms between {W}eyl spaces and twistorial maps {II}.
\newblock {\em Ann. Inst. Fourier (Grenoble)}, 60(2):433--453, 2010.

\bibitem[Ou97]{ou-97}
Ye-Lin Ou.
\newblock Quadratic harmonic morphisms and {O}-systems.
\newblock {\em Ann. Inst. Fourier (Grenoble)}, 47(2):687--713, 1997.

\bibitem[OW96]{ou-wood-96}
Ye-Lin Ou and John~C. Wood.
\newblock On the classification of quadratic harmonic morphisms between
  {E}uclidean spaces.
\newblock {\em Algebras Groups Geom.}, 13(1):41--53, 1996.

\bibitem[Pal79]{palais-79}
Richard~S. Palais.
\newblock The principle of symmetric criticality.
\newblock {\em Comm. Math. Phys.}, 69(1):19--30, 1979.

\bibitem[PT86]{palais-terng-86}
Richard~S. Palais and Chuu-Lian Terng.
\newblock Reduction of variables for minimal submanifolds.
\newblock {\em Proc. Amer. Math. Soc.}, 98(3):480--484, 1986.

\bibitem[Rie25]{riedler-polynomials-23}
Oskar Riedler.
\newblock Polynomial harmonic morphisms and eigenfamilies on spheres.
\newblock {\em Math. Z.}, 311(3):Paper No. 56, 25, 2025.

\bibitem[Tay23]{taylor-3}
Michael~E. Taylor.
\newblock {\em Partial differential equations {III}. {N}onlinear equations},
  volume 117 of {\em Applied Mathematical Sciences}.
\newblock Springer, Cham, third edition, 2023.

\bibitem[Ura00]{urakawa-00}
Hajime Urakawa.
\newblock A discrete analogue of the harmonic morphism and {G}reen kernel
  comparison theorems.
\newblock {\em Glasg. Math. J.}, 42(3):319--334, 2000.

\bibitem[Wan94]{wang-minimal}
Qi~Ming Wang.
\newblock On a class of minimal hypersurfaces in $\mathbb{R}^n$.
\newblock {\em Math. Ann.}, 298(2):207--251, 1994.

\end{thebibliography}
\end{document}